\documentclass[11pt,reqno]{amsart}
\usepackage{amsfonts, amssymb, amsmath, dsfont,tikz}
\usepackage[linktoc=page, colorlinks, linkcolor=blue, citecolor=blue]{hyperref}
\usepackage{graphicx, caption}
\usetikzlibrary{patterns}
\usepackage{enumerate}
\numberwithin{equation}{section}

\DeclareMathOperator{\E}{\mathbb{E}}

\def \P {\mathbb{P}}
\def \R {\mathbb{R}}

\def \Z {\mathbb{Z}}

\def \EE {\mathcal{E}}
\def \AA {\mathcal{A}}

\def \MM {\mathcal{M}}
\def \NN {\mathcal{N}}

\def \BB {\mathcal{B}}

\def \eps {\varepsilon}

\def \ind {{\mathds 1}}

\newtheorem{theorem}{Theorem}[section]
\newtheorem{proposition}[theorem]{Proposition}
\newtheorem{corollary}[theorem]{Corollary}
\newtheorem{lemma}[theorem]{Lemma}
\newtheorem{conjecture}[theorem]{Conjecture}

\theoremstyle{remark}
\newtheorem{remark}[theorem]{Remark}

\begin{document}

\title{Constructive regularization of the random matrix norm}
\author{Elizaveta Rebrova}


\thanks{Affiliation: Department of Mathematics, University of California, Los Angeles (rebrova@math.ucla.edu). Significant part of the work was done when E.R. was a Ph.D. student at the University of Michigan. Work supported in part by Allen Shields Memorial Fellowship.}

\begin{abstract}

We show a simple local norm regularization algorithm that works with high probability. Namely, we prove that if the entries of a $n \times n$ matrix $A$ are i.i.d. symmetrically distributed and have finite second moment, it is enough to zero out a small fraction of the rows and columns of $A$ with largest $L_2$ norms in order to bring the operator norm of $A$ to the almost optimal order $O(\sqrt{\log \log n \cdot n})$. As a corollary, we also obtain a constructive procedure to find a small submatrix of $A$ that one can zero out to achieve the same goal.

This work is a natural continuation of our recent work with R. Vershynin, where we have shown that the norm of $A$ can be reduced to the optimal order $O(\sqrt{n})$ by zeroing out just a small submatrix of $A$, but did not provide a constructive procedure to find this small submatrix.

Our current approach extends the norm regularization techniques developed for the graph adjacency (Bernoulli) matrices in the works of Feige and Ofek, and Le, Levina and Vershynin to the considerably broader class of matrices.
\end{abstract}

\maketitle

\section{Introduction} \label{introduction}

What should we call an \emph{optimal} order of an operator norm of a random $n \times n$ matrix? If we consider a matrix $A$ with independent standard Gaussian entries, then by the classical Bai-Yin law (see, for example, \cite{Tao}) 
$$ 
  \|A\|/\sqrt{n} \to 2 \quad \text{ almost surely, }
$$
as the dimension $n \to \infty$. Moreover, the $2\sqrt{n}$ asymptotic holds for more general classes of matrices. By \cite{Bai-Yin-Kri}, if the entries of $A$ have zero mean and bounded fourth moment, then
$$
\|A\| = (2+o(1)) \sqrt{n}
$$
with high probability. If we are concerned to get an explicit (non-asymptotic) probability estimate for all large enough $n$, an application of Bernstein's inequality (see, for example, in \cite{V, V-HDP}) gives 
$$\P\{\|A\| \le t \sqrt{n}\} \ge 1 - e^{-c_0t^2n} \quad \text{ for } t \ge C_0$$
for the matrices with i.i.d. subgaussian entries. Here, $c_0, C_0 > 0$ are absolute constants. The non-asymptotic extensions to more general distributions are also available, see \cite{Seginer, Latala, BvH, vH}. 

Also, note that the order $\sqrt{n}$ is the best we can generally hope for. Indeed, if the entries of $A$ have variance $C$, then the typical magnitude of the Euclidean norm of a row of $A$ is $ \sim \sqrt{n}$, and the operator norm of $A$ cannot be smaller than that. So, it is natural to assume $O(\sqrt{n})$ as the ``ideal order'' of the operator norm of an $n \times n$ i.i.d. random matrix.

However, if we do not assume that the matrix entries have four finite moments, we do not have ideal order $O(\sqrt{n})$: the weak fourth moment is necessary for the convergence in probability of $\|A\|/\sqrt{n}$ when $n$ grows to infinity (see \cite{Silv}). Moreover, for the matrices with the entries having two finite moments, an explicit family of examples, constructed in \cite{LS}, shows that $A$ can  have $\|A\| \sim O(n^{\alpha})$ for any $\alpha \le 1$ with substantial probability.

\bigskip

This motivates the following questions: what are the obstructions in the structure of $A$ that make its operator norm too large? Under what conditions and how can we regularize the matrix restoring the optimal $O(\sqrt{n})$ norm with high probability? Clearly, interesting regularization would be the one that does not change $A$ too much, for example, that changes only a small fraction of the entries of $A$. We call such regularization \emph{local}.

The first question was answered in our previous work with R. Vershynin (\cite{ReV}). We have shown that one can enforce the norm bound $\|A\| \sim \sqrt{n}$ by modifying the entries in a small submatrix of $A$ if and only if the i.i.d. entries of $A$ have zero moment and finite variance. The proof strategy was to construct a way to regularize $\|.\|_{\infty \to 2}$ norm of $A$, and to apply a form of Grothendieck-Pietsch theorem (see \cite[Proposition 15.11]{LT}) to claim that some additional small correction regularizes the operator norm $\|A\|$. This last step made it impossible to find the submatrix explicitly. 

\bigskip

In the current work we give an (almost optimal) answer to the remaining constructiveness question, namely, \emph{when local regularization is possible, how to fix the norm of $A$ by a small change to the optimal order?} The main result of the paper is 

\begin{theorem}[Constructive regularization] \label{main}
Let $A$ be a random $n~\times~n$ matrix with i.i.d. entries $A_{ij}$ having symmetric distribution such that $\E A_{ij}^2 = 1$. Then for any $\eps \in (0, 1/6]$, $r \ge 1$ with probability $1 - n^{0.1-r}$ the following holds: if we replace with zeros at most $\eps n$ rows and $\eps n$ columns with largest $L_2$-norms (as vectors in $\R^n$), then the resulting matrix  $\tilde{A}$ will have a well-bounded operator norm
\begin{equation}\label{main norm estimate}
\|\tilde A\| \le C r\sqrt{c_{\eps}  n \cdot \ln \ln n}.
\end{equation}
Here $c_{\eps}  =(\ln \eps^{-1})/\eps$ and $C > 0$ is a sufficiently large absolute constant.
\end{theorem}

\begin{remark}\label{regularization by row norm}
Typically, all the rows and columns of the matrix $\tilde{A}$ have $L_2$-norms bounded by $O(\sqrt{c_{\eps}  n})$. One way to check this is via the inconstructive regularization result proved in~\cite{ReV}. Indeed, with probability $1 - 7\exp(- \eps n/12)$, removing some $\eps n \times \eps n$ sub-matrix of $A$, we get a matrix $\bar A$ such that $\|\bar A\| \lesssim \sqrt{c_{\eps}  n}$ (\cite[Theorem~1]{ReV}). It implies that all the rows and columns of $\bar A$ have well-bounded $L_2$-norms (of order at most $\sqrt{c_{\eps} n}$). Since all but $\eps n$ rows and $\eps n$ columns of $A$ coincide with those of $\bar A$, there can be at most $\eps n$ rows and columns in $A$ having larger $L_2$-norms. Thus, regularization described in the statement of Theorem~\ref{main} zeros out them all.

Moreover, the proof Theorem~\ref{main} holds without changes if we define $\tilde A$ as the result of zeroing out of all rows and columns having $L_2$-norm bigger than $C\sqrt{c_{\eps}n}$. As we just discussed, this is an even more delicate change of the matrix $A$ with high probability.
\end{remark}

The regularization procedures discussed above (in Theorem~\ref{main} and Remark~\ref{regularization by row norm}) are local, as they change only a small fraction of the matrix entries. However, they still change more than $\eps n \times \eps n$ submatrix as promised by \cite[Theorem~1.1]{ReV}. As a corollary of Theorem~\ref{main}, we also obtain a polynomial algorithm that regularizes the norm of $A$ with high probability by zeroing out its \emph{small submatrix}. 

This algorithm addresses separately subsets of matrix entries having similar magnitude. We define these subsets via order statistics of i.i.d. samples $A_{ij}$: let $\hat A_1, \ldots, \hat A_{n^2}$ be the non-increasing rearrangement of the entries $A_{ij}$ (in sense of absolute values, namely, $|\hat A_1| \ge \ldots \ge |\hat A_{n^2}|$). Then, 
\begin{equation} \label{aal}
\AA_l := \{\hat A_{ \lceil 2^{l-1} n \eps + 1 \rceil }, \ldots, \hat A_{ \lceil 2^{l} n \eps \rceil }\} \quad \text{ for any } l \in \Z_{\ge 0}. \end{equation}
We are ready to state submatrix regularization algorithm:
\vspace*{0.2cm}

\begin{center}
  \begin{tabular}{l}
    \hline
    \textbf{Algorithm 1: Local norm regularization} \\ \hline
\textbf{Input}: matrix $A = (A_{ij})_{i,j = 1}^n$, constants $\eps, c_{\eps} > 0$, positive integer $l_{\max}$, \\ 
    \;\;\;\;\;\;\;\;\;\;\; disjoint entry subsets $\AA_l$ defined by \eqref{aal} for $l \le l_{\max}$\\
\textbf{Output}: $\tilde A$ - $n \times n$ matrix, regularized version of $A$ \\
    \hline

 \textbf{1.} Zero out $\lceil n \eps/2 \rceil$ entries $A_{ij}$ with the largest absolute values;\\
 \textbf{2.} For $l = 0, \ldots, l_{\max}$ find column index subset $J_l$ in the following way:\\
    \;\;\;  \textbf{2a.} For $j \in [n]$ define $e_j^{row}(\AA_l) := |\{i: A_{ij} \in \AA_l \}|$;\\
 \;\;\;    \textbf{2b.} For every $i, j \in [n]$ define the weight \\
 \;\;\; \;\;\;  \; $
W^l_{ij} := 
\begin{cases}
1,  &\text{if } e_j^{row}(\AA_l)  \le c_\eps np_l\text{ or } A_{ij} \notin \AA_l, \\
c_\eps np_l/e_j^{row}(\AA_l),  &\text{otherwise,}
\end{cases}
$ \\
  \;\;\; \;\;\;  \; where we denoted $p_l = 2^{l}\eps/n$;
\\
 \;\;\;  \textbf{2c.}  Then, define $J_l := \{j: \prod_{i=1}^n W^l_{ij} \le 0.1\}$; \\
  \textbf{3.} Find subset $\hat J$ of $n\eps/4$ indices corresponding to the columns of $A$ \\
  \;\;\;   with the largest $L_2$-norms, define $J := (\cup_l J_l) \cup \hat J$; \\
 \textbf{4.} Repeat Steps 2-3 for $A^T$ to find row subset $I := (\cup_l I_l) \cup \hat I$; \\
  \textbf{5.} Zero out all the entries of $A$ in the product subset $I \times J$ to get $\tilde A$. \\
 \hline
  \end{tabular}
\end{center}

If the matrix $A$ is taken from the same model as in Theorem~\ref{main}, the regularization provided by Algorithm~1 finds $\eps n \times \eps n$ submatrix $I \times J$ that one can replace with zeros to get a matrix $\tilde A$ with well-bounded norm. This is proved in the following

\begin{corollary}[Constructive regularization, submatrix version] \label{subblock regularization}
Let $A$ be a random $n~\times~n$ matrix with i.i.d. entries $A_{ij}$ having symmetric distribution such that $\E A_{ij}^2 = 1$.  Let $\eps \in (0, 1/6]$, $c_{\eps}  =(\ln \eps^{-1})/\eps$ and $l_{\max} = \lfloor\log_2 (\ln n/\ln \eps^{-4})\rfloor$. Subsets $\AA_l$ are defined by \eqref{aal} for $l = 0, \ldots, l_{\max}$.

Suppose that matrix $\tilde A$ is constructed from $A$ by Algorithm~1. Then with probability $1 - n^{0.1-r}$ matrix $\tilde A$ differs from $A$ on at most $\eps n \times \eps n$ submatrix, and 
$$
\|\tilde A\| \le C r^{3/2} \sqrt{c_{\eps}  n \cdot \ln \ln n},
$$
where  and $C > 0$ is a sufficiently large absolute constant.
\end{corollary}

The fact that the sub-matrix regularization algorithm is more involved than the one presented in Theorem~\ref{main} is somewhat natural. Zeroing out a small  submatrix must still bring the $L_2$-norms of all rows and columns to the order $O(\sqrt{n})$. Since the majority of rows and columns stays untouched in such regularization, essentially, one needs to find the most ``dense" part of the matrix. 

The procedure of assigning weights to the matrix entries  row-wise, multiplying them to set column weights, and then thresholding columns with the low weights is a delicate way to do so. This weight construction was originally used in \cite{ReT, ReV} for the matrices with i.i.d. scaled Bernoulli entries. Here we employ the same construction to regularize the entries at every ``level" independently (by definition, $k$-th ``level" contains the entries of $A$ that belong to $2^{-k}$-quantile of the distribution of $A_{ij}^2$). Additionally, to make the algorithm distribution oblivious, we estimate quantiles by order statistics of the matrix entries (since a random matrix naturally contains $n^2$ samples of the distribution $\xi \sim A_{ij}$). The idea to estimate quantiles of some distribution by the order statistics of a set of samples is both natural and well-known in the statistics literature (see, e.g. \cite{ZZ, OrderStatistics}).

\subsection{Notations and structure of the paper}

We use the following standard notations throughout the paper. Positive absolute constants are denoted $C, C_1, c, c_1$, etc. Their values 
may be different from line to line. We often write $a \lesssim b$ to indicate that $a \le C b$ for some absolute constant $C$.

The discrete interval $\{1,2,\ldots,n\}$ is denoted by $[n]$. Given a finite set $S$, by $|S|$ we denote its cardinality. The standard inner product in $\R^n$ shall be denoted by $\langle\cdot,\cdot\rangle$. Given $p\in[1,\infty]$, $\|\cdot\|_p$ is the standard $\ell_p^n$-norm in $\R^n$. 
Given a matrix $A$, $\|.\|$ denotes the operator $l_2 \to l_2$ norm of the matrix:
$$
\|A\| := \max_{x \in S^{n-1}} \|Ax\|_2.
$$

We write $row_1(A), \ldots, row_n(A) \in \R^m$ to denote the rows of any $m \times n$ matrix $A$ and $col_1(A), \ldots, col_m(A) \in \R^n$ to denote its columns. We are going to use sparsity of the matrices in the proof. We denote by $e_i^{row}(A)$ the number of non-zero entries in the $i$-th row of the matrix $A$, also $e_i^{col}(A)$ denotes the number of non-zero entries in the $i$-th column of $A$.

The rest of the paper is structured as follows. The proof of Theorem~\ref{main} is based on the previously known regularization results developed for Bernoulli random matrices (mainly in the works of Feige and Ofek \cite{FO}, and Le, Levina and Vershynin \cite{LV}). In Section~\ref{bernoulli_section} we review some results specific to the Bernoulli matrices and briefly explain how they will be used later in the text. In Section~\ref{middle section} we show how to extend the Bernoulli techniques to more general class of matrices and prove central Proposition~\ref{middle entries}. In Section~\ref{proof section} we combine these techniques to conclude the proof of Theorem~\ref{main}. In Section~\ref{corollary section} we prove Corollary~\ref{subblock regularization}, and the last Section~\ref{discussion section} contains discussion of the results and related open questions.

\subsection*{Acknowledgement}

The author is grateful to Roman Vershynin for the suggestion to look at the work of Feige and Ofek, helpful and  encouraging discussions, as well as comments related to the presentation of the paper. The author is also grateful to Konstantin Tikhomirov for mentioning the idea to estimate quantiles of the entries distribution from their order statistics, which made Algorithm~1 more elegant.

\section{Auxiliary results for Bernoulli random matrices}\label{bernoulli_section}

General idea of the proof of Theorem~\ref{main} is to split the entries of the matrix $A$ into subsets of entries having similar absolute values, and bound them from above by the properly scaled $0$-$1$ Bernoulli variables. Then we use some known regularization results that hold for Bernoulli random matrices for each subset separately. 

The goal of this section is to review several useful results related to the regularization of the norms of Bernoulli matrices. 

\subsection{Regularization of the norms of Bernoulli random matrices}
Consider a $n \times n$ Bernoulli matrix $B$ with independent $0$-$1$ entries such that $\P\{B_{ij} = 1\} = p$. Since the second moment of its entries $\E B_{ij}^2 \sim p$, from the facts discussed in the beginning of Section~\ref{introduction}, one would expect an ideal operator norm $\|B\| \sim \sqrt{np}$.

This is exactly what happens with high probability when success probability $p$ is large enough ($p \gtrsim \sqrt{\ln n}$, see, e.g. \cite{FO}). If $p \ll \sqrt{\ln n}$ the norm can stay larger than optimal (\cite{KS}). However, it is known that the regularization procedure described in Theorem~\ref{main} works in the case of Bernoulli matrices, and moreover, results in the optimal norm order $\|\tilde B \| \sim \sqrt{np}$. Since all non-zero entries in $B$ have the same size, the $L_2$ norm description can be simplified in terms of the number of non-zero entries in each row or column. 

Namely, Feige an Ofek proved in \cite{FO} that if we zero out all rows and columns that contain more than $C np$ non-zero entries, then the resulting matrix satisfies $\|\tilde B \| \sim \sqrt{np}$. This result was improved by Le, Levina and Vershynin. In \cite{LV}, the authors demonstrate that it is enough to zero out any part of the rows and columns with too many non-zeros, or reweigh them in any way, to satisfy $e^{row}_i \le Cnp$ and $e^{col}_i \le Cnp$ for any $i \in [n]$ to obtain the resulting matrix  $\|\tilde B \| \sim \sqrt{np}$.

\subsection{Regularization and the quadratic form}
So, zeroing out all rows and columns that contain more than $C np$ non-zero entries regularizes the quadratic form
\begin{equation} \label{normsum}
\|B\| = \sup_{u, v \in S^{n-1}}  \big|\sum_{ij} B_{ij} u_i v_j\big|
\end{equation}
 to the optimal order $\sqrt{np}$. 

We will need the following Lemma~\ref{tail_conds}, addressing the part of the sum \eqref{normsum} (over the indices $i,j$ such that $\{|u_iv_j| \ge \sqrt{p/n}\}$). It was first proved by Feige and Ofek in \cite{FO} and later appears in \cite{CRV}. We give a short sketch of its proof here for the sake of completeness and common notations. Let us also emphasize that even though for the regularization procedure introduced in \cite{FO} it is crucial to zero out a \emph{product} subset of the entries, in the framework of the Lemma~\ref{tail_conds} it is possible to zero out \emph{any} subset of the entries.

\begin{lemma} \label{tail_conds}
Let $B$ be a  $n \times n$ Bernoulli matrix with independent $0$-$1$ entries such that $\P\{B_{ij} = 1\} = p$. Let $r \ge 1$. Let $\BB \subset [n] \times [n]$ be an index subset, such that if we zero out all $B_{ij}$ with $(i,j) \notin \BB$, then every row and column of $B$ has at most $C_0rpn$ non-zero entries. Then with probability $1 - n^{-r}$
$$
\sup_{u, v \in S^{n-1}} \sum_{{\begin{subarray}{c}(i,j) \in \BB:\\
    |u_iv_j| \ge \sqrt{p/n}\end{subarray} }} B_{ij} |u_i v_j| \le Cr \sqrt{np},
    $$
    where $C$ is a large enough absolute constant.

\end{lemma}

The proof of Lemma~\ref{tail_conds} is based on a technical Lemma~\ref{tail_conds_sq_form} stated below. The proof of Lemma~\ref{tail_conds_sq_form} is completely deterministic and can be found in \cite{FO, CRV}.
 
   \begin{lemma}\label{tail_conds_sq_form}\cite[Lemma 21]{CRV}
Let $B$ be an $n \times n$ matrix with $0$-$1$ elements. Let $p > 0$, such that every row and column of $B$ contains at most $C_0np$ ones.  For index subsets  $S, T \subset [n]$ define 
$$e(S,T):= \sum_{i \in S, j \in T} B_{ij}$$
 (i.e. number of non-zero elements in the submatrix spanned by $S \times T$). Suppose that for any $S, T \subset [n]$ one of the following conditions holds: 
 \begin{enumerate}
  \item[(A)] $ e(S, T) \le C_1 |S||T| p $, or
  \item[(B)] $ e(S, T) \cdot \log\left(\frac{e(S, T)}{ |S||T| p}\right) \le C_2 |T| \log\left(\frac{n}{|T|}\right),$
 \end{enumerate}
  with some constants $C_1$ and $C_2$ independent from $S, T$ and $n$.
 Then for any $u, v \in S^{n-1}$ 
 $$ \sum_{i,j: |u_i v_j| \ge \sqrt{p/n}} B_{ij} |u_iv_j|  \le C \sqrt{n p},$$
 where the constant $C  = \max\{ 16, 3C_0, 32C_1, 32C_3\}$.
   \end{lemma}

\begin{proof}(of Lemma~\ref{tail_conds}) In view of Lemma~\ref{tail_conds_sq_form} it is enough to show that with probability $1 - n^{-r}$ for every $S, T \subset [n]$ and 
$$e_{\BB}(S,T)= \sum_{(i,j) \in \BB \cap (S\times T)} B_{ij}$$
 one of the conditions $(A)$ and $(B)$ holds. Without loss of generality let us assume that $|T| \ge |S|$. 
 
 If $|T| \ge n/e$, we have
$
 e_{\BB}(S, T) \le |S|\cdot C_0rpn \le C_0rpe|T||S|.
 $
Hence, condition $(A)$ holds with $C_1 = C_0re$.

If both $|S|, |T| < n/e$, then $\P\{\EE\} \le n^{-r}$ for the event 
$$
\EE := \big\{\exists S,T: |S|, |T| < n/e, \, e_{\BB} (S,T) > l_{S,T} p |S| |T|\big\}
$$
and $l_{S,T}$ being a number such that
$$
 l_{S,T} \ln l_{S,T} := \ln \frac{n}{|T|} \cdot \frac{3(r+6) |T|}{p|S||T|}.
$$
 Indeed, the probability estimate for $\EE$ follows from the Chernoff's inequality applied to the sum of independent variables $e(S,T) \ge e_{\BB}(S,T)$, combined with the Stirling formula estimating the number of proper sets $S$ and $T$, and the fact that the function $f(x) = (x/n)^x$ is monotonically increasing on $[1, n/e]$ (see \cite[Section~2.2.5]{FO} for the computation details).
 
Thus, with probability $1 - n^{-r}$, for any $S, T$ , such that $|S|, |T| < n/e$, condition $(B)$ holds:
\begin{align*}
 e_{\BB}(S, T) \cdot \ln\left(\frac{e_{\BB}(S, T)}{ |S||T| p}\right)  &\le l_{S,T} p |S| |T|\cdot \ln l_{S,T} \\
 &\le 3(r+6) |T| \ln \frac{n}{|T|}.
\end{align*}
This concludes Lemma~\ref{tail_conds} with $C  = \max\{ 32C_0re, 100(r+6)\}$.
\end{proof}

\subsection{Decomposition of Bernoulli matrices}\label{bernoulli section}

The idea is to apply an approach developed for Bernoulli matrices for the truncations of the entries of $A$, having absolute values on the same ``level", and then to sum over these ``levels". However, there is an obstacle: there is no simple way to see which rows and columns in the Bernoulli ``levels" have too many non-zeros. Even if we know that the rows and columns of the general matrix $A$ are well-bounded, some of the ``levels" might be too large if the others are small enough.

To address this issue without making regularization procedure more complex, we are going to use an additional structural decomposition for Bernoulli random matrices, first shown in the work of Le, Levina and Vershynin~\cite{LV}. The next proposition is a direct corollary of \cite[Theorem~2.6]{LV}:

\begin{proposition}[Decomposition lemma]\label{decomposition lemma}
Let $B$ be a  $n \times n$ Bernoulli matrix with independent $0$-$1$ entries such that $\P\{B_{ij} = 1\} = p$. Then, for any $n \ge 4$, $p \ge 0$ and $r \ge 1$, with probability at least $1 - 3n^{-r}$ all entries of $B$ can be divided into three disjoint classes $[n] \times [n] = \BB_1 \sqcup \BB_2 \sqcup \BB_3$, such that 
\begin{itemize}
\item $e^{row}_{i}(\BB_1) \le C_1 r^3np$ and $e^{col}_{i}(\BB_1) \le C_1 r^3np$
\item $e^{row}_{i}(\BB_2) \le C_2r$
\item $e^{col}_{i}(\BB_3) \le C_2r,$
\end{itemize}
where $e^{col/row}_{i}(\BB)$ is the number of non-zero elements in $i$-th row or column of $B$ belonging to the class $\BB$ and $C_1$, $C_2$ are absolute constants.
\end{proposition}

\begin{remark}\label{r_order}
Following the same methods as was employed in the proof of \cite[Theorem~2.6]{LV}, one can check that Proposition~\ref{decomposition lemma} actually holds with linear (instead of cubic) dependence on $r$, namely  $e^{row}_{i}(\BB_1) \le C_1 rnp$ and $e^{col}_{i}(\BB_1) \le C_1 rnp$.
\end{remark}

\section{From Bernoulli to general matrices}\label{middle section}
The goal of this section is to prove Proposition~\ref{middle entries}, that provides a way to generalize the regularization results known for Bernoulli random matrices to the general case:

\begin{proposition}\label{middle entries}
Suppose $A$ is a random $n \times n$ matrix with i.i.d. symmetric entries $A_{ij}$ with $\E A_{ij}^2 = 1$. Let $\tilde A$ be the resulting matrix after we zeroed out row and column subsets of $A$ in any way, such that 
\begin{equation}
\|row_{i}(\tilde{A})\|^2_2 \le c_{\eps} n \text{ and } \|col_{i}(\tilde{A})\|^2_2 \le c_{\eps} n
\end{equation}
for all $i = 1, \ldots, n$. Let
$$
\tilde M := \tilde A \cdot \ind_{\{|\tilde A_{ij}| \in (2^{k_0}, 2^{k_1}]\}}, \text{ and } k_1 - k_0 =: \kappa.
 $$
Then with probability at least $1 - 10\kappa n^{-r}$ we have
$$
\|\tilde M\| \le C r \sqrt{c_\eps n \kappa}.
$$
Here $c_\eps$ is any positive $\eps$-dependence and $C > 0$ is an absolute constant.
\end{proposition}

Let us first collect several auxiliary lemmas, that will be used in the proof of Proposition~\ref{middle entries}.
  \begin{lemma}\label{light_members}
Consider a $n \times n$ random matrix $M$ with independent symmetric entries and $\E M^2_{ij} \le 1$. Consider two vectors $u = (u_i)_{i = 1}^n$ and $v = (v_j)_{j =1}^n$ such that $u, v \in S^{n-1}$. Denote the event 
$$ \MM^{light}_{ij} = \{|M_{ij}| |u_iv_j| \le 2/\sqrt{n}\}$$ 
and let $Q \subset [n] \times [n]$ be an index subset.
Then for any constant $C \ge 3$
$$ | \sum_{i j} u_i M_{ij} \ind_{\{(i,j) \in Q\}} \ind_{\MM^{light}_{ij}} v_j | \le C\sqrt{n} $$
with probability at least $1 - 2\exp(-Cn/2)$.
\end{lemma}

\begin{proof}
Let $R_{ij} := M_{ij}\ind_{\{(i,j) \in Q\}} \ind_{\MM^{light}_{ij}}$. Note that $R_{ij}$ are centered due to the symmetric distribution of $M_{ij}$, and they are independent as $M_{ij}$ are. So we can apply Bernstein's inequality for bounded distributions (see, for example, \cite[Theorem 2.8.4]{V-HDP}) to bound the sum:
 $$
 \P\{ | \sum_{ij} u_i R_{ij}  v_j| \ge t\}  \le 2 \exp\left( - \frac{t^2/2}{\sigma^2 + Kt/3}\right), 
 $$
 where 
 $$K = \max_{i,j} |u_i R_{ij} v_j| \le 2/\sqrt{n} \quad \text{ and } \quad \sigma^2 = \sum_{ij} \E(u_i R_{ij} v_j)^2.$$
 Note that  $ \E R^2_{ij} \le  \E M^2_{ij}$, as $R^2_{ij} \le M^2_{ij}$ almost surely, and $\E M^2_{ij} \le 1$. So,
$$
\sigma^2 = \sum_{ij} u^2_i \E R^2_{ij} v^2_j \le \sum_{i,j} u_i^2 v_j^2 = 1,
$$
as $\sum_i u_i^2 = \sum_j v_j^2 = 1$. So, taking $t = C\sqrt{n}$, we obtain
\begin{align*}
 \P\{ | \sum_{(i,j)} u_i M_{ij} \ind_{\{(i,j) \in Q\}} &\ind_{\MM^{light}_{ij}} v_j |  \ge C \sqrt{n}\} \le 2 \exp(-Cn/2) 
\end{align*}
for any $C \ge 3$. This concludes the statement of the lemma.
 \end{proof}

The following lemma is a version of \cite[Lemma~3.3]{LV}.
\begin{lemma}\label{l2const}
  For any matrix $Q$ and vectors $u,v \in S^{n-1}$, we have
  $$
  \sum_{ij} Q_{ij} u_{i} v_{j} \le \ \max_j \|col_j(Q)\|_2 \cdot (\max_{i} e_i^{row}(Q) )^{1/2}.
  $$
\end{lemma}
\begin{proof}
Indeed, 
\begin{align*}
\sum_{ij} Q_{ij} u_{i} v_{j} 
&\le \sum_j v_j (\sum_{i: Q_{ij} \ne 0} Q_{ij} u_i) \\
&\le \sum_j v_j (\sum_{i: Q_{ij} \ne 0} Q^2_{ij})^{1/2} (\sum_{i: Q_{ij} \ne 0} u^2_i)^{1/2} \quad (*) \\
&\le \max_j \|col_j(Q)\|_2  \sum_j v_j (\sum_{i: Q_{ij} \ne 0} u^2_i)^{1/2}\\
&\le \max_j \|col_j(Q)\|_2  (\sum_j v^2_j)^{1/2} (\sum_j \sum_{i: Q_{ij} \ne 0} u^2_i)^{1/2} \quad (*) \\
&\le \max_j \|col_j(Q)\|_2 \cdot 1\cdot (\sum_i u^2_i\sum_{j: Q_{ij} \ne 0} 1)^{1/2} \quad \text{(since $\|v\|_2 = 1$)} \\
&\le \max_j \|col_j(Q)\|_2 (\sum_i u^2_ie_i^{row}(Q))^{1/2}\\
&\le \max_j \|col_j(Q)\|_2 \cdot (\max_{i} e_i^{row}(Q) )^{1/2}. \quad \text{(since $\|u\|_2 = 1$)}
\end{align*}
Steps (*) hold by the Cauchy-Schwarz inequality. Lemma~\ref{l2const} is proved.
\end{proof}

\bigskip

In the proof of Proposition~\ref{middle entries} we are going to use a standard splitting ``by size" of a non-negative random variable. Let $X = 0$ or $X \in (2^{k_0}, 2^{k_1}]$ almost surely. Then, clearly,
$$
X \le \sum_{k = k_0+1}^{k_1} 2^k \ind_{\{X \in (2^{k-1}, 2^k]\}}.
$$
If additionally $\E X^2 \le 1$ and we denote
\begin{equation}\label{pk}
	p_k := \P\{|M_{ij}| \in (2^{k-1}, 2^{k}]\},
\end{equation}
the following estimates hold for $p_k$. First, the sum
\begin{equation}\label{pksumestimate}
\sum p_k 2^{2k} \le 4 \sum p_k 2^{2(k-1)} \le 4\E X^2 \le 4. 
\end{equation}

\bigskip

Now we are ready to prove Proposition~\ref{middle entries}.
\subsection{Proof of Proposition~\ref{middle entries}}

\textbf{Step 1. Net approximation.}

 Let $\NN$ be a $1/2$-net on $S^{n-1}$ with cardinality $|\NN| \le 5^n$ (the existence of such net is a standard fact that can be found, e.g. in \cite{V-HDP}). We will use a simple net approximation of the norm (see, e.g., \cite[Lemma~4.4.1]{V-HDP}), namely,
$$\|\tilde M \| \le 4 \max_{u, v \in \NN} \langle \tilde Mu, v\rangle =  4 \max_{u, v \in \NN} |\sum_{ij} \tilde M_{ij}u_i v_j|.$$
We will split the sum into two parts and bound each of them separately, based on the absolute value of the element. 

Let $M := A \cdot \ind_{\{|A_{ij}| \in (2^{k_0}, 2^{k_1}]\}}$. For any fixed $u, v \in \NN$ and every $i, j \in [n]$ we define an event 
$$\MM^{light}_{ij} = \{|M_{ij}||u_iv_j| \ge 2/\sqrt{n}\}.$$ 
Then,
\begin{align*}
\max_{u, v \in \NN} &|\sum_{i,j} \tilde M_{ij}u_i v_j| \\
&\le \max_{u, v \in \NN} |\sum_{i,j} \tilde M_{ij} (\ind_{\MM^{light}_{ij}} + \ind_{(\MM^{light}_{ij})^c}) u_i v_j|\\
   &\le \max_{u, v \in \NN} |\sum_{i,j} \tilde M_{ij}\ind_{\MM^{light}_{ij}} u_i v_j| + \max_{u, v \in \NN} |\sum_{ij} \tilde M_{ij}\ind_{(\MM^{light}_{ij})^c}u_i v_j|.
\end{align*}

\bigskip

 \textbf{Step 2. Light members.}

By Lemma~\ref{light_members}, for any fixed $u, v \in S^{n-1}$ and a fixed subset of indices $Q$ (assuming that $Q^c$ is a set of rows and columns to delete), 
\begin{equation}\label{est_light}
| \sum_{i,j} u_i M_{ij} \ind_{\{(i,j) \in Q\}} \ind_{\MM^{light}_{ij}} v_j | \le 12\sqrt{n} 
\end{equation}
with probability at most $2\exp(-6n)$. Now, taking union bound over $5^n$ choices for $u \in \NN$, as many choices for $v \in \NN$, and $2^{2n}$ choices for the row and column subset $Q^c$, we obtain that 
\begin{equation}\label{eq: light_part}
\P \{ | \sum_{i,j} u_i \tilde M_{ij}  \ind_{\MM^{light}_{ij}} v_j |\le 12\sqrt{n}\} \ge 1 - 2\exp(-n).
\end{equation}

\bigskip
 \textbf{Step 3. Other members.}
 
 The second sum can be roughly bounded by the sum of absolute values of its members:
 \begin{align*}
    |\sum_{i,j} \tilde M_{ij}&\ind_{(\MM^{light}_{ij})^c}u_i v_j| \\
    &\le \sum_{i,j} |\tilde M_{ij}|\ind_{(\MM^{light}_{ij})^c}|u_i v_j| \\
  &\le \sum_{i,j}\big[\sum_{k = k_0+1}^{k_1}2^k  \ind_{\{|\tilde M_{ij}| \in (2^{k-1}, 2^k]\}}\big] \ind_{\{|M_{ij}||u_iv_j| \ge 2/\sqrt{n}\}}|u_i v_j| 
  \end{align*}
  Note that as long as $\ind_{\{|\tilde M_{ij}| \in (2^{k-1}, 2^k]\}} = 1$ we also have that $|M_{ij}| \le 2^k$. Indeed, $|M_{ij}| > 2^k$ implies either $|\tilde M_{ij}| > 2^k$ or $|\tilde M_{ij}| = 0$. In any case, $|\tilde M_{ij}| \notin (2^{k-1}, 2^k]$. So, the last expression is bounded above by
$$
  \sum_{i,j}\sum_{k = k_0+1}^{k_1}2^k  \ind_{\{|\tilde M_{ij}| \in (2^{k-1}, 2^k]\}} \ind_{\{2^k|u_iv_j| \ge 2/\sqrt{n}\}}|u_i v_j| $$

 Since $\E M_{ij}^2 \le \E A_{ij}^2 = 1$, from \eqref{pksumestimate} we have $2^{1 - k} \ge \sqrt{p_k}$ for any $k$, where $p_k$ is probability of the $k$-th level (defined in \eqref{pk}). 
As a result, in Step 3 we got
\begin{align}\label{second term}
  |\sum_{i,j} \tilde M_{ij}&\ind_{(\MM^{light}_{ij})^c} u_i v_j| \nonumber\\
  &\le \sum_{k = k_0+1}^{k_1} 2^k \sum_{i,j:|u_iv_j| \ge \sqrt{p_k/n}}  \ind_{\{|\tilde M_{ij}| \in (2^{k-1}, 2^k]\}}  |u_i v_j|.  
\end{align}

\bigskip
 \textbf{Step 4. Bernoulli matrices.}
 For each ``size" $k = k_0+1, \ldots, k_1$ let us define a $n \times n$ matrix $B^k$ with independent Bernoulli entries
 $$B^k_{ij}:= \ind_{\{|M_{ij}| \in (2^{k-1}, 2^k]\}}, \quad \E B^k_{ij} = p_k.
 $$
  By Decomposition Lemma~\ref{decomposition lemma} (and Remark~\ref{r_order}), with probability $1 - 3n^{-r}$, the entries of every $B^k$ can be assigned to one of three disjoint classes: $\BB_1^k$, where all rows and columns have at most $C_1r p_k n$ non-zero entries; $\BB_2^k$, where all the columns have at most $C_2 r$ non-zero entries; and $\BB_3^k$, where all the rows have at most $C_2 r$ non-zero entries. We are going to further split the sum \eqref{second term} into three sums containing elements of these three classes, and estimate each of them separately.

\begin{itemize}
\item[$\BB_1:$]
 The part with the entries from $\BB_1^k$   satisfies the conditions of Lemma~\ref{tail_conds}.
 For any $k = k_0 + 1, \ldots, k_1$
 $$
 \sum_{\begin{subarray}{c}(i,j) \in \BB_1^k:\\
    |u_iv_j| \ge \sqrt{p_k/n}\end{subarray} }  \ind_{\{|\tilde M_{ij}| \in (2^{k-1}, 2^k]\}}  |u_i v_j| \le \sum_{{\begin{subarray}{c}(i,j) \in \BB_1^k:\\
    |u_iv_j| \ge \sqrt{p_k/n}\end{subarray} }} B_{ij}^k |u_i v_j|.
$$
By Lemma~\ref{tail_conds}, this sum is bounded by $C r\sqrt{p_k n}$ with probability at least $1 - n^{-r} $. So, for any $u, v \in S^{n-1}$
$$
\sum_{k = k_0+1}^{k_1} 2^k \sum_{{\begin{subarray}{c}(i,j) \in \BB_1^k:\\
    |u_iv_j| \ge \sqrt{p_k/n}\end{subarray} }} B_{ij}^k |u_i v_j| \le Cr\sqrt{n} \sum_{k = k_0+1}^{k_1} 2^k \sqrt{p_k}.
$$
Then, by the Cauchy-Schwarz inequality and estimate \eqref{pksumestimate},
$$
 \sum_{k = k_0+1}^{k_1} 2^k \sqrt{p_k}  \le \sqrt{\sum_{k = k_0+1}^{k_1}  2^{2k} p_k} \sqrt{\sum_{k = k_0+1}^{k_1}  1} \le 2\sqrt{\kappa}.
$$

\item[$\BB_2:$]
The part with the entries from $\BB_2^k$ can be estimated by Lemma~\ref{l2const}. We have that
$$
 \sum_{k = k_0+1}^{k_1} 2^k \sum_{{\begin{subarray}{c}(i,j) \in \BB_2^k:\\
    |u_iv_j| \ge \sqrt{p_k/n}\end{subarray} }}  \ind_{\{|\tilde M_{ij}| \in (2^{k-1}, 2^k]\}}  |u_i v_j| \le \sum_{i,j} Q_{ij} |u_i v_j|, 
$$
where
$$
Q_{ij} := \sum_{k = k_0+1}^{k_1} 2^k \ind_{\{(i,j) \in \BB_2\}} \ind_{\{|\tilde M_{ij}| \in (2^{k-1}, 2^k]\}} 
$$
Note that for every fixed $i = 1, \ldots, n$, number of non-zero entries $Q_{ij}$ in the row $i$ is at most $C_2r \kappa$. Also, $|Q_{ij}| \le 2 |\tilde M_{ij}| \le 2 |\tilde A_{ij}|$ almost surely, so maximum $L_2$-norm of the column of $Q$ is $\sqrt{c_{\eps} n}$. By Lemma~\ref{l2const}, this implies that for any $u, v \in S^{n-1}$
$$\sum_{i,j} Q_{ij} |u_i v_j| \le \sqrt{C_2 r \kappa c_{\eps} n}$$

\item[$\BB_3:$] The part with the entries from $\BB_3^k$ can be estimated in the same way as $\BB_2^k$, repeating the argument for $A^T$.

\end{itemize}

\textbf{Step 5. Conclusion.}
Now we can combine the estimates obtained for the light members~\eqref{eq: light_part} and all three parts of the non-light sum, to get that 
$$
\|\tilde M \| \le 12 \sqrt{n} + C r \sqrt{\kappa n}  + 2\sqrt{C_2 r \kappa c_{\eps} n} \lesssim  r \sqrt{c_{\eps} \kappa n } 
$$
with probability at least $1 - 2 e^{-n} - 3\kappa n^{-r} - 6 n^{-r} \ge 1 - 10 n^{-r}\kappa$ for all $n$ large enough. Proposition~\ref{middle entries} is proved.

\section{Conclusions: proof of Theorem~\ref{main} and further directions} \label{proof section}
In this section we conclude the proof of Theorem~\ref{main}. As we have seen in the previous section, splitting the entries of the matrix $A$ into $\kappa$ ``levels" with the similar absolute value produces extra $\sqrt{\kappa}$ factor in the norm estimate. Hence, we care to make the number of levels as small as possible. We are going to show that this number can be as small as $C\ln \ln n$, where $n$ is the size of the matrix. The reason is that we need only to consider the ``average" entries of the matrix, those with the absolute values between $O(\sqrt{n / \ln n})$ and $O(\sqrt{n})$. The ``large" entries will be all replaced by zeros in regularization, and restriction to the ``small" entries creates a matrix with the optimal norm (no regularization is needed). One way to check this is applying the following result of Bandeira and Van Handel:

\begin{theorem}\label{theor: bvh}\cite[Lemma 21]{CRV} 
Let $X$ be an $n \times n$ matrix whose entries $X_{ij}$ are independent centered random variables. For any $\eps \in (0, 1/2]$ there exists a constant $c_{\eps}$ such that for every $t \ge 0$
$$
  \P\{\|X\| \ge (1 + \eps) 6\sigma + t \} \le n \exp( - t^2 / c_{\eps} \sigma_*^2),
$$
where $\sigma$ is a maxium expected row and column norm:
$$ 
\sigma^2 := \max(\sigma^2_1, \sigma^2_2), \text{ where } \quad \sigma^2_1 = \max_i \sum_j \E (X_{ij}^2), \quad  \sigma^2_2 = \max_j \sum_i \E (X_{ij}^2);
$$
and $ \sigma_*$ is a maximum absolute value of an entriy:
$$
 \sigma_* := \max_{ij} \|X_{ij}\|_{\infty}.
$$

\end{theorem}

\begin{lemma}\label{small entries}
Suppose $S$ is a random $n \times n$ matrix with i.i.d. mean zero entries $S_{ij}$, such that $\E S_{ij}^2 \le 1$ and $|S_{ij}| < \bar c\sqrt{n}/\sqrt{\ln n}$. Let $r \ge 1$.
If $\bar c < c$, then with probability at least $1 - n^{-r}$
$$
\|S\| \le 13r \sqrt{n}.
$$
Here $c$ is a small enough absolute constant.
\end{lemma}

\begin{proof}
The proof follows directly from Theorem~\ref{theor: bvh} with $t = r\sqrt{n}$ and $\eps = 1$. It is enough to take $c = 1/11c_1$, where $c_1$ is a constant from the statement of Theorem~\ref{theor: bvh}. 

\end{proof}

Now we are ready to prove Theorem~\ref{main}.

\subsection{Proof of Theorem~\ref{main}}

 Let us decompose $A$ into a sum of three $n \times n$ matrices
\begin{equation}\label{general split}
A := S + M + L,
\end{equation}
where $S$ contains the entries of $A$ that satisfy $|A_{ij}| \le 2^{k_0}$, the matrix $M$ contains the entries for which $2^{k_0} < |A_{ij}| \le 2^{k_1}$,
and $L$ contains large entries, satisfying $|A_{ij}| > 2^{k_1}$. Here, 
$$k_0 := \left\lfloor \frac{1}{2} \log_2 \frac{c_1 n}{\ln n} \right\rfloor, \quad k_1:=  \left\lceil \frac{1}{2}\log_2 (C_2 c_\eps n)\right\rceil, \, \text{ where } c_\eps = (\ln \eps^{-1})/\eps
$$ 
and $c_1, C_2 > 0$ are absolute constants. 

Note that $S$, $M$ and $L$ inherit essential properties of the matrix $A$. First, they also have i.i.d. entries (since they are obtained by independent individual truncations from the i.i.d. elements $A_{ij}$). Due to the symmetric distribution of the entries of $A$, the entries of $S$, $M$ and $L$ have mean zero. Also, their second moment is bounded from above by $\E A_{ij}^2 = 1$. 

Note that all the entries in $S$ satisfy $|A_{ij}| \le \sqrt{c_1 n/\ln n}$. Thus, as long as we choose constant $c_1$ small enough to satisfy the condition in Lemma~\ref{small entries}, the norm of $S$ can be estimated as
\begin{equation}\label{norm1}
 \P \{\|S\| > 13 \sqrt{n}\} < n^{-r}.
\end{equation}
Clearly, replacing by zeros some rows and column subset can only decrease the norm of $S$.

By Remark~\ref{regularization by row norm}, with probability at least $ 1 - 7 \exp(\eps n/12)$, all rows and columns of $\tilde{A}$ have bounded norms: for $i = 1, \ldots, n$
$$\|row_i(\tilde{A})\|_2 \le C \sqrt{c_\eps  n} \text{ and }  \|col_i(\tilde{A})\|_2 \le C \sqrt{c_\eps n}.$$
In particular, it implies that all the entries of $\tilde A$ have absolute values bound by $C \sqrt{c_\eps n}$. So, by taking constant $C_2 \ge C^2$, we achieve that $L$ will be empty after the regularization.

 Proposition~\ref{middle entries} estimates the norm of $M$ after the regularization (zeroing out row and columns with large norms):
\begin{equation}\label{norm2}
\P \{\|\tilde A \cdot \ind_{\{|\tilde A_{ij}| \in (2^{k_0}, 2^{k_1}]\}}\| > C c_{\eps} r \sqrt{n \kappa} \} \le 10n^{-r}\kappa.
\end{equation}
By definition,
$$
\kappa := k_1 - k_0 \le \frac{1}{2} \left[\log_2 (C_2 c_\eps n) - \log_2 \frac{c_1n }{\ln n} + 2\right] \le \log_2 \ln n
$$
for all large enough $n$.

Using triangle inequality to combine norm estimates~\eqref{norm1} and \eqref{norm2}, we get $\| \tilde A\| \lesssim c_{\eps} r \sqrt{n \cdot \ln \ln n }$
with probability at least 
$$
1- n^{-r} - 7e^{-\eps n/12} - 10n^{-r}\ln\ln n \ge 1 - n^{-r +0.1}
$$ 
for all $n$ large enough. This concludes the proof of Theorem~\ref{main}.

\begin{remark}\label{regularization by row norm_s}
Note that the only place in the argument where it matters what entry subset we replace by zeros is Step~2 of the proof of Proposition~\ref{middle entries}. To have a manageable union bound estimate, we need to be sure that the number of options for the potential subset to be deleted is of order $\exp(n)$ (so, $\exp(\ln 2 \cdot n^2)$ options for a general entry subset would be too much). 

Hence, also recalling Remark~\ref{regularization by row norm}, we emphasize that the norm estimate \eqref{main norm estimate} holds with the probability $1 - n^{0.1- r}$ as long as we achieve $L_2$-norm of all rows and columns bounded by $C\sqrt{c_\eps n}$  by zeroing out any product subset of the entries of $A$. 

\end{remark}

\section{Proof of Corollary~\ref{subblock regularization}}\label{corollary section}

General idea of the proof of Corollary~\ref{subblock regularization} is to show that after the regularization procedure all rows and columns of the matrix have well-bounded $L_2$-norms, and then apply Theorem~\ref{main}.

Originally, the core part of the algorithm (Step 2) was presented for Bernoulli random matrices in \cite{ReV, ReT}. We will use the following version of \cite[Lemma~5.1]{ReV} (based on the ideas developed in \cite[Proposition~2.1]{ReT}):
\begin{lemma}\label{constructive column cut}
 Let $B$ be a $n \times n$ matrix with independent $0$-$1$ valued entries, $\E B_{ij} \le p$. Then for any $ L \ge 10$ with probability $1 - \exp(-n \exp(-L pn))$
  the following holds. If we define 
   $$
W_{ij} := 
\begin{cases}
1,  &\text{if } e_j^{row}(B)  \le L np \text{ or } B_{ij} = 0, \\
L np/e_j^{row}(B) ,  &\text{otherwise.}
\end{cases}
  $$
  and $V_j := \prod_{i = 1}^n W_{ij}$ , and $J := \{j: V_j < 0.1\}$, then
  $$
  |J| \le n \exp(-L n p) \quad \text{ and } \, \sum_{j \in J^c} B_{ij} \le 10Lnp, \text{ for any } i \in [n] .
  $$
\end{lemma}

In order to pass from Bernoulli case to the general distribution case we are going to use some version of ``level truncation" idea once again. Note that here we need the probabilities of the levels $p_l$ to be both not too large (for the joint cardinality estimate) and not too small (for the probability estimate union bound). This motivates the idea to group ``similar size" entries $A_{ij}$ not by absolute value $|A_{ij}|$, but by common $2^{-l}$-quantiles of the distribution of $\xi \sim A_{ij}^2$.

\begin{remark}
A version of Corollary~\ref{subblock regularization} can be proved, when one would define the sets $\AA_l$ to contain all $A_{ij}$ such that $A_{ij}^2 \in (q_{k-1}, q_l]$,  where $q_l$ is $2^l$-th quantile of the distribution of $A_{ij}^2$, namely, 
\begin{equation}\label{quant}
q_l := \inf\{t: \P \{A^2_{ij} > t\} \le 2^{-l}\}.
\end{equation}

 The proof of this version is actually almost identical to the one presented below and gives smaller absolute constants. However, an additional requirement \emph{to know quantiles} of the distribution of the entries in order to regularize the norm of the matrix seems undesirable. So we are going to prove the distribution oblivious version as presented by Algorithm~1.
\end{remark}

The next lemma shows that the order statistics used in the statement of Corollary~\ref{subblock regularization} approximate quantiles of the distribution of $A_{ij}^2$.
\begin{lemma} [Order statistics as approximate quantiles]
Let $\hat A_1, \ldots \hat A_{n^2}$ be all the entries of $n \times n$ random matrix $A$ in an non-increasing order and $q_k$ be $2^{-k}$ quantiles of the distribution of $A_{ij}^2$ defined by \eqref{quant}.
 Then with probability at least $1 - 4\exp(-n^2 2^{-k_1 - 2})$  for all $k = 1, \ldots, k_1$
\begin{equation}\label{approx quantile estimate}
 q_{k-2} \le \hat A^2_{ \lceil n^2 2^{1-k}\rceil } \le q_k.
 \end{equation}
\end{lemma}

\begin{proof}
A direct application of Chernoff's inequality shows that for any $k$ 
$$
\P \{ \nu_1 > \lceil n^2 2^{1-k} \rceil\} \le \exp(-n^2 2^{-k}/4),
$$where $\nu_1$ is a number of entries $A_{ij}$ such that $A_{ij}^2 > q_k$.
Another application of Chernoff's inequality lower bound shows that
$$
\P \{ \nu_2 < \lceil n^2 2^{1-k}\rceil \} \le \exp(-n^2 2^{-k}/4),
$$where $\nu_2$ is a number of entries $A_{ij}$ such that $A_{ij}^2 > q_{k-2}$.

Then, with probability at least $1 - 2 \exp(-n^2 2^{-k})$ the order statistics $\hat A^2_{n^2 2^{2-k}}$ is at least $q_{k-2}$ and at most $q_k$. Taking union bound, equation \eqref{approx quantile estimate} holds with probability $1 - 4 \exp(-n^2 2^{-k_1}/4)$ for all $k = 1, \ldots, k_1$.
\end{proof}
\begin{remark}\label{quantile estimation remark}
We will use $k_1 = \lceil \log_2 (8 n/\eps)\rceil$. An easy computation  using \eqref{approx quantile estimate} shows that 
$$q_{k_1 - l - 3} \le \hat A^2_{\lceil2^{4 + l - k_1} n^2 \rceil} \le \hat A^2_{\lceil 2^l n \eps\rceil} \le \hat A^2_{\lceil2^{3 + l - k_1} n^2 \rceil} \le q_{k_1- l - 2}$$ 
for all $l = 0, \ldots,  l_{\max}$ with probability $1 - 4 \exp(-n\eps/4)$. Then, for $\AA_l$ as defined in \eqref{aal} and all $l \le l_{\max}$,
\begin{equation}\label{plest}
\P\{A_{ij} \in \AA_l\} \le 2^{3 + l - k_1} \le 2^{l}\eps/n.
\end{equation}

\end{remark}

Now we are ready to prove Corollary~\ref{subblock regularization}.

\subsection{Proof of Corollary~\ref{subblock regularization}}
Let $k_1 = \lceil \log_2 (8 n/\eps)\rceil$, and $q_{k_1}$ is a corresponding $2^{-k_1}$ quantile of the distribution of $A_{ij}^2$ (as defined by~\eqref{quant}). It is easy to check by Chernoff's inequality that the total number of entries in $A$ such that $A_{ij}^2 \ge q_{k_1}$ is at most $\eps n/2$ with probability at least $1 - e^{-\eps n/4}$. So, all these ``large" entries will be replaced by zeros at the Step 1 of regularization Algorithm~1.

To prove Corollary~\ref{subblock regularization}, it is enough to show that:
\begin{enumerate}
\item Algorithm~1 makes all rows and columns of the truncated matrix $A \cdot \ind_{\{A_{ij}^2 < q_{k_1}\}}$ have norms bounded by $C\sqrt{c_\eps n}$. Then, in view of Remark~\ref{regularization by row norm_s}, we can apply Theorem~\ref{main} to conclude the desired norm estimate. 
\item cardinalities of the exceptional index subsets $|I|, |J| \le \eps n/2$ with high probability. Then the regularization procedure is indeed \emph{local}.
\end{enumerate}
The matrix $\bar A := A \cdot \ind_{\{A_{ij}^2 < q_{k_1}\}}$ is naturally decomposed into union of $l_{\max}$ ``levels" with the entries coming from sets $\AA_l$ and ``the leftover" part that contains $A_{ij}$ such that $A^2_{ij} < \hat A^2_{\lceil n \eps 2^{l_{\max}} \rceil}$. So,
$$
\bar A_{ij} = A_{ij} \ind_{\{A_{ij} \in \cup \AA_l\}} + A_{ij} \ind_{\{|A_{ij}| < \hat A_{\lceil n \eps 2^{l_{\max}}  \rceil}\}} =: A^{Large} + A^{Small}.
$$

 All the rows and columns of $A^{Small}$ have $L_2$-norms at most $C \sqrt{n r c_\eps}$ with probability at least $1 - n^{-r}$ (without any regularization). This follows from an application of Bernstein inequality (e.g., \cite[Theorem 2.8.4]{V-HDP}) for a sum of independent centered entries bounded by $\sqrt{Cn/\ln n}$. Indeed, we just need to check boundedness condition. Recall that $l_{\max} = \lfloor\log_2 (\ln n/\ln \eps^{-4})\rfloor$. By definition of quantiles $q_k$ and Markov's inequality
 $$
 \P\{A_{ij}^2 \ge q_{k_1 - 2 - l_{\max}} \} \ge 2^{-k_1 + 1} \frac{\ln n}{ \ln \eps^{-4}} \ge \P\{ A_{ij}^2 \ge \frac{32 c_\eps n}{\ln n}\}.
$$
Hence, the entries of $A^{Small}$ can be estimated from above 
$$
\hat A^2_{\lceil n \eps 2^{l_{\max}} \rceil} \le q_{k_1 - 2 - l_{\max}} \lesssim \frac{32 c_\eps n}{\ln n}.
$$

\bigskip

 For the part $A^{Large}$ we use Lemma~\ref{constructive column cut} applied to $n\times n$ matrices with i.i.d. entries $B^l_{ij} = \ind_{\{A_{ij} \in \AA_l\}}$ for each $l = 0, \ldots, l_{\max}$ with $L = 2c_{\eps}$ and $p_l = 2^{l}\eps/n$ (which is a valid choice by Remark~\ref{quantile estimation remark}). From the union bound estimate we can conclude that the statement of Lemma~\ref{constructive column cut} holds  for all $l \le l_{\max}$ with high probability
$$
1 - \sum_{l \le l_{\max}} \exp(-n \exp(-2c_{\eps} n 2^{l}\eps/n) )\ge 1 - \exp(-n^{0.5}).
$$

Recall that $\bar J = \cup_{l} J_l$ is the union of all exceptional column index subsets found for all matrices $A\cdot \ind_{\{A_{ij} \in \AA_l\}}$ with $l = 0, \ldots, l_{\max}$. Note that by the definition of quantiles and second moment condition, 
\begin{equation}\label{secmom}
\sum_{s=0}^{\infty} q_s 2^{-s-1} \le \E A_{ij}^2 \le 1.
\end{equation}
By Lemma~\ref{constructive column cut} we can estimate for every $i \in [n]$
\begin{align*}
\|row_i(A^{Large}_{[n] \times \bar J^c})\|_2^2 &\le \sum_{l \le l_{\max}} q_{k_1 - l - 2} 20 c_\eps n p_l \\
&\le \sum_{l \le l_{\max}} q_{k_1 - l - 2} 20 c_\eps n \frac{2^{l - k_1} \eps}{n} 2^{k_1} \le 160 c_{\eps} n,
\end{align*}
as $2^{k_1} \le 16 n/\eps$, we used \eqref{secmom} with $s = k_1 - l - 2$ in the last step.

Then, by the $L_2$-norm triangle inequality applied to the rows of $A_{[n] \times \bar J^c}^{Large}$ and $A^{Small}$, we have the row boundedness condition satisfied for $\bar A_{[n] \times \bar J^c}$. Next, on Step 3 we add the set $\hat J$ of columns with largest $L_2$-norms. The same argument as in Remark~\ref{regularization by row norm} shows that with probability at least $1- n^{-r}$ there are no columns with the norm larger than $C \sqrt{c_\eps n}$ outside the set $\hat J$. So, matrix $\tilde A_{[n] \times J^c}$ has all rows and columns norms well-bounded (recall that $J := \bar J \cup \hat J$). Then, by Theorem~\ref{main}, with high probability $1 - C (\ln \ln n)n^{- r}$
\begin{equation}\label{1}
\|\tilde A_{[n] \times J^c}\| \lesssim r^{3/2} \sqrt{c_{\eps} n \ln\ln n}.
\end{equation}
Repeating the same argument for the transpose, we have that  
 \begin{equation}\label{2}
 \|\tilde A_{I^c \times J}\| \le \|\tilde A_{I^c \times [n]}\| \lesssim r^{3/2} \sqrt{c_{\eps} n \ln\ln n}.
 \end{equation}
 Now we can combine \eqref{1} and \eqref{2} by triangle inequality for the operator norm to conclude the desired norm estimate for $\tilde A$ on the intersection of good events, namely, with probability 
 $$1 - \exp^{-\eps n/4} - n^{-r} - \exp(-n^{0.5}) -  2C (\ln \ln n)n^{- r} \ge 1 - n^{0.1-r}.$$

Finally, let us check that the regularization is local. Again by Lemma~\ref{constructive column cut}, the total number of exceptional columns
$$
|\bar J| = |\bigcup_l J_l| \le \sum_{l \le l_{\max}} n \exp(-2c_\eps n 2^{l}\eps/n) \le n \eps/4,
$$
since we are summing a geometric progression and $l \ge 0$. Since the same argument holds for the cardinality of $\bar I$, we can conclude that with high probability Algorithm~1 makes changes only in a $n \eps \times n \eps$ submatrix of $A$. This concludes the proof of Corollary~\ref{subblock regularization}.

\section{Discussion} \label{discussion section}

\subsection*{Regularization by the individual corrections of entries.}

Do we actually need to look at the rows and columns of $A$? A simpler and very intuitive idea would be to regularize the norm of $A$ just by zeroing out a few large entries of $A$. However, this approach does not work for the case when the entries have only two finite moments: for the efficient local regularization, one has to account for the mutual positions of the entries in the matrix, not only for their absolute values. 

Only in the case when $A_{ij}$ have \emph{more} than two finite moments the truncation idea works and it is not hard to derive the following result from known bounds on random matrices such as \cite{vH, Seginer, Auff} (see also the discussion in \cite[Section~1.4]{ReV}).

In the two moments case, individual correction of the entries can guarantee a bound with bigger additional factor $\ln n$ in the norm. It can be derived from known general bounds on random matrices, such as the matrix Bernstein's inequality (\cite{Tropp}). One would apply the matrix Bernstein's inequality for the entries truncated at level $\sqrt{n}$ to get that $\| \tilde A \| \le \eps^{-1/2} \sqrt{n} \cdot \ln n$.

We consider Theorem~\ref{main} more advantageous with respect to the individual corrections approach not only because we are able to bound the norm closer to the optimal order $\sqrt{n}$, but also due to the fact that it gives more adequate information about the obstructions to the ideal norm bound. Namely, they are not only in the entries that are too large, but also in the rows and columns that accumulate too many entries (all of which are, potentially, of average size).

\subsection*{Symmetry assumption}
An assumption that the entries of $A$ has to have symmetric distribution does not look natural and potentially can be avoided. We need it in the current argument to keep zero mean after various truncations by absolute value (in \eqref{general split} and also in \eqref{est_light}). The standard symmetrization techniques (see [\cite{LT}, Lemma 6.3]) would not work in this case since we combine convex norm function with truncation (zeroing out of a product subset), which is not convex. 

\subsection*{Dependence on $n$}
Another potential improvement is an extra $\sqrt{\ln \ln n}$ factor on the optimal $n$-order $\|\tilde{A} \| \sim \sqrt{n}$. The reason for its appearance in our proof is that we consider restrictions of $A$ to the discretization ``levels" independently, and independently estimate their norms. The second moment assumption gives us that $\sum 2^{2k} p_k \sim 1$. However, the best we can hope for a norm of one ``level" (after proper regularization) is $2^k \sqrt{n p_k}$ (since this is an expected $L_2$-norm of a restricted row). Thus, we end up summing square roots of the converging series, $\sum (2^{2k} p_k)^{1/2}$, which for some distributions is as large as square root of the number of summands ($\ln \ln n$ in our case). 

It would be desirable to remove extra $\sqrt{\ln \ln n}$ term and symmetric distribution assumption, proving something like the following
\begin{conjecture}
  Consider an $n \times n$ random matrix $A$ with i.i.d. mean zero entries such that $\E A_{ij}^2 = 1$. Let $\tilde{A}$ be the matrix that obtained from $A$ by zeroing out all rows and columns such that
  \begin{equation}\label{deletion}
   \|row_i(A)\|_m \ge C\E \|row_i(A)\|_m, \quad
   \|col_i(A)\|_m \ge C\E \|col_i(A)\|_m
  \end{equation}
for some $L_m$-norm to be specified (e.g. $m = 2$). Then with probability $1  - o(1)$ the operator norm satisfies $\|\tilde{A}\| \le C' \sqrt{n}$.
\end{conjecture}

Note that this result would be somewhat similar to the estimate proved by Seginer (\cite{Seginer}): in expectation, the norm of the matrix with i.i.d. elements is bounded by the largest norm of its row or column. However, note that after cutting ``heavy" rows and columns we lose independence of the entries in the resulting matrix. And in general, the question of the norm regularization is not equivalent to another interesting question about the sufficient conditions on the distribution of the entries that ensure an optimal order of the operator norm.

\end{document}